\documentclass[11pt]{amsart}

\usepackage{amsmath}
\usepackage{amssymb}
\usepackage{graphicx}
\usepackage{xcolor}
\usepackage[T1]{fontenc}
\usepackage{calligra}

\newtheorem{theorem}{Theorem}[section]
\newtheorem{lemma}[theorem]{Lemma}
\newtheorem{corollary}[theorem]{Corollary}

\theoremstyle{definition}
\newtheorem{definition}[theorem]{Definition}
\newtheorem{example}[theorem]{Example}

\theoremstyle{remark}

\numberwithin{equation}{section}

\def\V{\Vert}

\def\ve{\vert\vert\vert}

\begin{document}

\title [Projection schemes  and the fixed point property]{ Projection schemes and the fixed point property in nonreflexive Banach spaces}

\author{T. Dom\'{\i}nguez Benavides; M. Jap\'on; P.K. Lin}

\address{Tom\'as Dom\'{\i}nguez Benavides\hfill\break
 Universidad de
Sevilla, Spain}
 \email{tomasd@us.es}
 \address{M. Jap\'on \hfill\break
Universidad de
Sevilla,  Spain}
 \email{japon@us.es}

 \address{P.K. Lin \hfill\break
University of Memphis,  USA}
 \email{pklin@memphis.edu}

\thanks{The first and second authors  are supported by  the Grant PID2023-148294NB-I00 funded by MICIU/AEI/ 10.13039/501100011033,   and the  projects “Programa operativo PPIT-FEDER Andalucía 2021-2027 SOL2024-31596 and SOL2024-31708, and María de Maeztu CEX2024-001517-M}

\keywords{Fixed point property, Nonexpansive mappings, Schauder basis, projection schemes, variable Lebesgue spaces}

\subjclass[2020]{46E30, 47H09, 47H10, 46B20}


\begin{abstract}
We introduce the notion of a projection scheme on a Banach space and develop an abstract framework for studying the fixed point property (FPP) for nonexpansive mappings. Within this framework, we obtain a general criterion ensuring the FPP for a broad class of nonreflexive Banach spaces. In particular, every uniformly convex Banach space admits an isometric embedding into a nonreflexive Banach space having the FPP. This criterion  unifies and extends several previously known fixed point results for nonreflexive Banach spaces with Schauder bases.

As an application of our general result, we characterize the FPP for variable Lebesgue spaces. More precisely, we prove that if  $(\Omega,\Sigma,\mu)$ is a $\sigma$-finite measure space and $p:\Omega \to [1,\infty]$ is a measurable function,  then $L^{p(\cdot)}(\Omega)$ has the FPP if and only if it does not contain an isometric copy of $\ell_1$. Consequently, as in the classical Lebesgue space framework, the presence of an isometric copy of $\ell_1$ completely characterizes the failure of the FPP in variable Lebesgue spaces. This result yields a class
of nonreflexive Banach  spaces with the fixed point property that contains isometric copies of infinitely many $L^p$-spaces, $1<p<\infty$.  Although nonreflexive sequence spaces with the FPP are known, to the best of our knowledge, these variable Lebesgue spaces provide the first examples of nonreflexive Banach function lattices with the FPP.

\end{abstract}

\maketitle

\section{Introduction}

A Banach space $X$ is said to have the fixed point property (FPP) if every nonexpansive mapping $T:C \to C$, defined on a nonempty, convex, closed, and bounded subset $C \subset X$, admits a fixed point. If the same conclusion holds whenever  $C$ is weakly compact, then $X$ is said to have the weak fixed point property (w-FPP). These two properties coincide for reflexive Banach spaces.

The study of fixed points for nonexpansive mappings is closely related to the geometry of Banach spaces. This connection dates back to the seminal work of Browder \cite{Br,Br2}, who proved that Hilbert spaces, and more generally uniformly convex Banach spaces, have the FPP. Since then, substantial effort has been devoted to identifying geometric conditions ensuring either the FPP or the w-FPP.

The classical nonreflexive spaces $\ell_1$, $c_0$, and $\ell_\infty$ fail the FPP, while it remains an open problem whether every reflexive Banach space has the FPP.  For many years, it was believed  that the FPP might characterize reflexivity. This was disproved by  P.K. Lin \cite{Lin-2008}, who constructed  an equivalent norm on   $\ell_1$   under which the renormed space has the FPP. Subsequently,  further renormings of $\ell_1$ and several other non-isomorphic-to-$\ell_1$ Banach spaces   were obtained with the FPP (but still hereditarily $\ell_1$) \cite{alfredo,todos, japon2, Lin-2010}.

\medskip

Our main goal is to investigate  the FPP beyond reflexivity. To this end, we will introduce the notion of a projection scheme, namely, a family $\mathcal{P}$ of projections  on a Banach space $X$ that abstracts some features of the canonical projections associated with a Schauder basis.    A projection scheme induces a locally convex vector topology  on $X$, together with a   corresponding notion of $\mathcal{P}$-convergence that reduces to coordinatewise convergence in the presence of the standard projections of a Schauder basis. 

These notions provide an abstract framework for developing new techniques to study the existence of fixed points of nonexpansive mappings, which is one of the main topics investigated throughout the paper.

\medskip

Section 2 contains preliminary facts related to fixed point theory and the basic definition of variable Lebesgue spaces. In particular, we will introduce the concept of asymptotic centers, approximate asymptotic centers and some features concerning the structure of asymptotic centers when the Banach space is uniformly convex. 

Section 3 is devoted to the introduction of the notion of a projection scheme $\mathcal{P}$ as a family of projections that imitates some properties of the resolution of the identity. Every Banach space can be equipped with a projection scheme consisting solely of the identity mapping. Under the existence of a Schauder basis, the canonical projections constitute a projection scheme according to our definition.  Further  examples of  $\mathcal{P}$-convergence and bounded completeness are    provided.

Sections 4 and 5 contain the core of the paper. In Section 4, we study asymptotic compactness of approximate asymptotic centers  in connection with $\mathcal{P}$-almost uniformly convex projection schemes (see Definitions \ref{kUC} and \ref{AUC}).  As a consequence, we prove that every convex closed bounded $T$-invariant set contains an approximate fixed point sequence that is  $\mathcal P$-convergent whenever $T$ is a nonexpansive mapping.  This constitutes the main ingredient in the proof of our abstract fixed point theorem in Section 5 (Theorem \ref{main}), where we establish a general criterion ensuring the fixed point property for a broad class of nonreflexive Banach spaces satisfying the structural assumptions introduced in the previous sections. This criterion unifies and substantially extends several previously known fixed point theorems for Banach spaces with Schauder bases \cite{todos}, while considerably enlarges the class of nonreflexive Banach spaces with the FPP, including now spaces are very far from being hereditarily $\ell_1$.  In particular, we prove that every uniformly convex Banach space admits an isometric embedding into a nonreflexive Banach space having the FPP.



The main application of our fixed point theorem is stated in Section 6, devoted to variable exponent Lebesgue spaces $L^{p(\cdot)}(\Omega)$. These spaces generalize the classical $L^p$ spaces by allowing the exponent to vary pointwise, and they arise naturally within the framework of modular function spaces introduced by Nakano and further developed by Orlicz and Musielak \cite{MO, N}. Their modern development has been strongly influenced by applications to models of electrorheological fluids \cite{Ru}, which has led to an extensive study of their analytic and geometric properties.

A complete characterization of the w-FPP for variable Lebesgue spaces was obtained in \cite{DJ}, showing that $L^{p(\cdot)}(\Omega)$ has the w-FPP if and only if it does not contain an isometric copy of $L^1([0,1])$.  Our main application establishes that,  similarly to  classical Lebesgue spaces,
 the presence of an isometric copy of $\ell_1$ is precisely the obstruction to the FPP in this setting. More precisely, we prove that $L^{p(\cdot)}(\Omega)$ has the FPP if and only if it does not contain an isometric copy of $\ell_1$. Moreover, this condition can be characterized in terms of the behavior of the exponent function $p$. Specifically, if $(\Omega,\Sigma,\mu)$ is a $\sigma$-finite measure space and $p:\Omega \to [1,\infty]$ is measurable, then $L^{p(\cdot)}(\Omega)$ has the FPP if and only if
\[
\operatorname*{ess\,sup}\{p(t): t\in\Omega,\ p(t)<\infty\}<\infty
\]
and the sets $\{t\in\Omega:\ p(t)=1\}$ and $\{t\in\Omega:\ p(t)=\infty\}$ contain only finitely many atoms.

In the final section, we present further applications of our main theorem. For instance,  we give an affirmative answer to the open question   raised in \cite{DJ} as to  whether certain nonreflexive Musielak-Orlicz sequence spaces $\ell^{p_n}$, with $\liminf p_n = 1$, have  the FPP. We also construct large classes of nonreflexive variable Lebesgue spaces having the FPP and containing infinitely many different $L^p([0,1])$ for $1<p<\infty$. To the best of our knowledge, our results provide the first examples of  nonreflexive Banach \textit{function} spaces with the FPP. Finally, we  note  that our abstract approach applies to a more general class of operators, including   mappings of type (C), introduced  by Suzuki \cite{Suz}, and, more generally,  mappings of $(L)$-type, introduced in \cite{LLMo}. In particular, all the  results in \cite{MoLL} concerning existence  of fixed point for mappings of $(L)$-type in nonreflexive spaces follow from Theorem \ref{main}.

\medskip

The abstract criterion developed in the article is formulated solely in terms of projection schemes and is independent of any specific Banach space structure. Consequently, once a Banach space admits a projection scheme satisfying our assumptions, the fixed point property follows automatically. We expect this framework to have further applications beyond the examples considered in the present paper.

\section{ Preliminaries: known results on fixed point theory and introduction of variable Lebesgue spaces  }\label{pre}

We begin by recalling some standard facts concerning the fixed point pro\-perty that will be used throughout the paper. For detailed information, see  \cite{GK, Kirk-Sims} and the references therein.

Let  $X$ be a  Banach space. We denote  by $B_X$ and $S_X$ the closed unit ball and the  unit sphere of $X$ respectively.

\medskip

\subsection{Fixed point preliminaries}

Let $X$  be a Banach space and $C$ be  a closed convex bounded subset of $X$ and   $T:C\to C$  a nonexpansive mapping, that is $\V Tx-Ty\V\le \V x-y\V$ for all $x,y\in C$.  Although $T $ needs not have a fixed point, one always has $\inf\{\V x-Tx\V: x\in C\}=0$.
A sequence $(x_n)\subset C$ verifying  $\lim_n\V T(x_n)-x_n\V=0$ is called an approximate fixed point sequence for $T$, a.f.p.s.   in short.

\medskip

We recall that a Banach space $X$ is said to be uniformly convex (UC) if for every $\epsilon>0$, there exists some $\delta>0$ such that if $\V x\V,\V y\V\le1$ and $\V x-y\V>\epsilon$, then
$$
\left\V \frac{x+y}{ 2}\right\V<1-\delta.
$$
 We next  recall the concept of  an asymptotic center.
Let $C$ be a convex, bounded, closed subset of $X$, $(x_n)\subset C$
and denote
$
r(\{x_n\},C):=\inf\{\limsup_n\V x_n-z\V: z\in C\}
$, the asymptotic radius of $(x_n)$ with respect to $C$. The set
$$
Z(\{x_n\},C):=\{z\in C: \limsup_n\V x_n-z\V=r(\{x_n\},C)\}
$$ is the asymptotic center of  $(x_n)$ with respect to $C$. Note that the asymptotic center may be empty  (see for instance \cite{KP}).

For every $\delta>0$, the set
$
Z(\{x_n\},C,\delta):=\{z\in C: \limsup_n\V x_n-z\V\le r(\{x_n\};C)+\delta\}
$
is always a nonempty, closed, convex subset of $C$ and $Z(\{x_n\},C)=\bigcap_{\delta>0}Z(\{x_n\},C,\delta)$. We call $Z(\{x_n\},C,\delta)$  the $\delta$-approximate asymptotic center of $(x_n)$ with respect to $C$. In particular, asymptotic centers are nonempty when $C$ is additionally weakly compact.
If $X$ is uniformly convex, then
$$\lim_{\delta\to 0^+} {\rm diam}\left[Z(\{x_n\},C,\delta)\right]=0,$$
which clearly shows that $Z(\{x_n\},C)$ is a singleton.  When the sequence $(x_n)$ is chosen to be an a.f.p.s. for  a nonexpansive mapping $T$, then $Z(\{x_n\},C)$ is $T$-invariant. If, in addition, the asymptotic center is a singleton $\{x\}$, then  $x\in C$ becomes a fixed point for $T$.

\subsection{ Nonreflexive Banach spaces with the FPP}

Whether every reflexive Banach space has the FPP is a long-standing open question dating back to the 1960s. In contrast, the first example of a nonreflexive Banach space with the FPP was obtained by P. K. Lin in 2008 \cite{Lin-2008}. He constructed an equivalent norm on $\ell_1$  under which  the FPP holds. This norm can be described as follows. Let $(\gamma_n)\subset (0,1)$ be a sequence with $\lim_n \gamma_n=1$.
Define
\begin{equation}\label{Linnorm}
\ve x\ve =\sup_n\gamma_n\sum_{k=n}^\infty |t_k|\quad \mbox{for } x=\sum_{n=1}^\infty t_n e_n.
\end{equation}

Since P.K. Lin's result, some other nonreflexive  Banach spaces have been discovered with the FPP (see for instance \cite{alfredo, todos, japon2, Lin-2010} or \cite[Section 4]{DJ}), but all of them are sequence spaces endowed with a Schauder basis for which the canonical projections  have some suitable properties.  In particular, they have
 the Schur property, which implies that the  spaces considered in these references are hereditarily $\ell_1$ \cite{alfredo}.

\medskip

If we focus on Lebesgue spaces $L^p(\Omega)$, which are uniformly convex when $1<p<+\infty$, it is deduced that  an infinite-dimensional $L^p(\Omega)$ has the FPP if and only if $p\ne 1,+\infty$ or, equivalently, $L^p(\Omega)$ has the FPP if and only if it does not contain $\ell_1$ isometrically.

\subsection{Variable Lebesgue spaces}
 As it was stated in the Introduction, one of the main goals of this article is to obtain a complete characterization of those variable Lebesgue spaces $L^{p(\cdot)}(\Omega)$ verifying the FPP.
We next briefly introduce these spaces which are a particular family of Orlicz modular spaces and which cover  Lebesgue spaces $L^p(\Omega)$ when the exponent function is constant.

Let
 $(\Omega,\Sigma,\mu)$ be a $\sigma$-finite measure space and $p:\Omega \to [1,+\infty]$ a measurable function. Denote $L_0(\Omega)$ the set formed by all measurable functions $f: \Omega\to \mathbb{R}$,
$\Omega_1:=\{t\in \Omega: p(t)=1\}$ and $\Omega_\infty:=\{t\in \Omega: p(t)=+\infty\}$.
Consider the function $\rho:L_0(\Omega)\to [0,\infty]$  defined by

\begin{equation}\label{mod}
\rho(f):=\int_{\Omega\setminus \Omega_\infty}|f(t)|^{p(t)} d\mu+  \, \text{ess sup} \{|f(t)|:t\in \Omega_\infty\}.
\end{equation}

 \medskip

For every $f,g\in L_0(\Omega)$, the function $\rho$ satisfies the following properties:

\medskip

$ (m1)$ $\rho(f)=0$ if and only if $f=0$;\\

$(m2)$ $ \rho(\alpha g)=\rho(g) $ for every scalar $ \alpha $ with $ |\alpha|=1 $;\\

$(m3)$  $ \rho(\alpha f+\beta g)\leq \alpha\rho(f)+\beta\rho(g) $ if $ \alpha+\beta=1 $ and $ \alpha, \beta \geq 0 $.

\medskip

A function $\rho$ satisfying properties $(m1)-(m3)$ is called a convex modular. The modular $ \rho $ defines the corresponding modular space, i.e. the vector space formed by all measurable functions $f: \Omega\to \mathbb{R}$ such that
$ \rho(f/\lambda)<\infty  \text{ for some  } \lambda >0. $

\medskip

The  Lebesgue space with variable exponent  $L^{p(\cdot)}(\Omega)$  (VLS in short) is defined as the above  modular   space  endowed with the Luxemburg norm associated to the modular $\rho$,  i.e.
\begin{align*}
\|f\| = \inf\left\{\alpha>0: \rho\left(\frac{f}{\alpha}\right)\leq 1\right\}.
\end{align*}

 Variable Lebesgue spaces form a class of Banach function spaces introduced by Nakano \cite{N}  and later developed by Orlicz and Musielak  \cite{MO}. Detailed information on these spaces can be found, for instance, in \cite{cruz,DHHR,Ru}.

It is well-known that $L^{p(\cdot)}(\Omega)$  is a  Banach function lattice whose    geometry  is strongly attached to the behaviour of the  exponent  function $p(\cdot)$.

We will follow the following standard notation: For any measurable set $E\in \Sigma$:
$$
 p_-(E)=\, \text{ess inf}_{t\in E }\, p(t), \quad p_+(E)=\, \text{ess sup}_{t\in E}\, p(t).
 $$
If $E=\Omega$ we just denote $p_-:=p_-(\Omega)$ and $p_+:=p_+(\Omega)$.

\medskip

We shall repeatedly use in Section 6, the following well-known properties \cite{cruz}.

\begin{lemma}\label{properties}   Let $(\Omega,\Sigma,\mu)$ be a $\sigma$-finite measure and  $p:\Omega\to [1,+\infty]$  an exponent function
 such that $p_+(\Omega\setminus\Omega_\infty)<\infty$.  Then, the following properties are all satisfied  for every $f\in L^{p(\cdot)}(\Omega)$:

(1)  $\rho(f/\Vert f\Vert)=1$. In particular $\rho(f)\leq 1$ if and only if $\Vert f\Vert\leq 1$

(2) The convergence in modulus for $\rho(\cdot)$ is equivalent to the norm convergence.





\end{lemma}

We will also use the following  convexity result \cite{TDB}:


\begin{lemma}\label{ucc2}
 Assume $1<p_-\leq p_+<\infty$. For  $2\geq \delta >0$ there exists $\alpha >0$ depending on $\delta, p^+$ and $p_-$ such that for any $u,v\in  L^{p(\cdot)}(\Omega)$ satisfying $\Vert u\Vert \leq 1$,  $\Vert v\Vert \leq 1$, $\Vert u-v\Vert \geq \delta$ we have
$$ \frac{\rho( u) +\rho( v)} {2} \geq  \rho \left( \frac {u+v}{2}\right) +\alpha.$$\end{lemma}

When  $\Omega$ is the set $\mathbb{N}$ of positive integers (so  $p:\mathbb{N}\to [1,\infty]$ is a sequence $\{p_n\}$) with the counting measure the space $L^{p(\cdot)}(\Omega)$ becomes   the Nakano space $\ell^{p_n}$.

\section{Projection schemes: definition and first properties}\label{PSsection}

In this section we introduce the concept of a projection scheme, inspired by the idea of a projectional resolution of the identity and designed to capture some of the properties of the canonical projections associated with  the Schauder basis,  used  by P.K. Lin in \cite{Lin-2008} for  $\ell_1$ and more generally in   \cite{todos}. We denote by $I:X\to X$ the identity mapping on a Banach space $X$.

\begin{definition}\label{PS} Let $(X,\V\cdot\V)$ be a Banach space.
Given a sequence of bounded linear projections $\mathcal{P}=\{P_n:X\to X: 1\le n\le \omega\}$,
we say that   $\mathcal{P}$ is a \textit{ Projection Scheme (PS)}   on $X$ if the following  conditions are satisfied:

\begin{itemize}

\item[a)]  For all $1\le n,m\le \omega$, $P_n\circ P_m=P_{\min\{n,m\}}$;

\item[b)] For every $x\in X$,  $\lim_n \V P_n(x)-P_\omega (x)\V=0$.

\item[c)] $I-P_\omega$ has finite rank.


\end{itemize}

\end{definition}

When $P_\omega=I$,
 condition $c)$ can be omitted. The  family $\mathcal{P}$ formed by $P_n=I$ for all $1\le n\le \omega$  is trivially
 a projection scheme on every Banach space.

Given a Banach space $X$ and a PS $\mathcal{P}=\{P_n: 1\le n\le \omega\}$ on $X$, the family $\{P_n: 1\le n\le \omega\}$ is uniformly bounded by the Banach-Steinhaus theorem.   We  denote  $K(\mathcal{P})=\sup_n\V P_n\V$  (and automatically  $\V P_\omega\V\le K(\mathcal{P})$). We  say that the projection scheme  $\mathcal{P}$ is monotone if $K(\mathcal{P})=1$. We say that $\mathcal{P}$  is premonotone if $\Vert I-P_n\Vert \le 1$ for every $n\in \mathbb{N}$ .


\begin{example}\label{Sch}

Let  $X$ be a Banach space endowed with a Schauder basis $(e_n)_n$.  The family $\mathcal{P}$ consisting of the canonical projections $P_n:X\to X$, defined   by $P_n(\sum_{k=1}^\infty x_k e_k)=\sum_{k=1}^n x_k e_k$  together with  $P_\omega=I$  is a projection scheme on $X$. \end{example}

\begin{example}\label{VLSPS}

Let $(\Omega,\Sigma,\mu)$ be a  $\sigma$-finite measure space.  Let $p: \Omega\to (1,+\infty)$ be a measurable function such that  $1=p_-(\Omega)<p_+(\Omega)<\infty$. Fix a nonincreasing sequence $(\gamma_n)\subset (1,+\infty)$  with $\lim_n\gamma_n=1$. For $n\in\mathbb{N}$, define $P_n: L^{p(\cdot)}(\Omega)\to L^{p(\cdot)}(\Omega)$ by $P_n(f)=f\cdot \chi_{p^{-1}(\gamma_n,\infty)}$ and $P_\omega=I$.  Then $\mathcal{P}=\{P_n: 1\le n\le \omega\}$ is a PS on $L^{p(\cdot)}(\Omega)$. More details, together with a slight modification of this construction, will be given in Section \ref{VLSsection}).

\end{example}

We now introduce the concept of $\mathcal{P}$-convergence on a Banach space $X$ equipped  with projection scheme $\mathcal{P}$.

\begin{definition}\label{Pconvergence} Let $\mathcal{P}=\{ P_k: 1\le k\le \omega\}$ be a projection scheme on  $X$.  We say that a  sequence $(x_n)_n$ in $X$ is $\mathcal{P}$-convergent to some $x\in X$, denoted by $\mathcal{P}-\lim_n x_n=x$,  if $\lim_n P_k(x_n)=P_k(x)$ for all $k\in\mathbb{N}$ and $\lim_n (I-P_\omega)(x_n)=(I-P_\omega)(x)$.
\end{definition}

The notion of  $\mathcal{P}$-convergence coincides with  convergence in the initial topology  on $X$ induced by the family of  mappings $\{P_n: n\in\mathbb{N}\}$ and $I-P_\omega$. Since these mappings are  linear and bounded, this  initial topology is generated by a  countable family of seminorms  which can be defined by $p_n(x):=\Vert P_n(x)\Vert $ for $n\in\mathbb{N}$ and $q(x):=\V(I-P_\omega)(x)\V$. The fact that every $x\in X$ can be written as $x=P_\omega x+(I-P_\omega)(x)$ shows that this family of seminorms separates points of $X$ and they generate a metrizable locally convex topological vector space whose topology is weaker than the norm topology.

When $P_n=I$ for all $n\in\mathbb{N}\cup\{\omega\}$,  $\mathcal{P}$-convergence coincides with  norm convergence. When $X$ is a  Banach space with a Schauder basis $(e_n)$ and  $\mathcal{P}$ is the PS described in Example \ref{Sch},  the  $\mathcal{P}$-convergence   is precisely  coordinatewise convergence.
For the PS described for $L^{p(\cdot)}(\Omega)$ in Example \ref{VLSPS}, it is not difficult to check that the topology generated by $\mathcal{P}$ is stronger than the topology of the convergence locally in measure.

\medskip
We conclude this section with the notion of a boundedly complete projection scheme, which is intended to capture the existence of a limit for bounded sequences, as in the case of a boundedly complete Schauder basis.

\begin{definition}\label{PS2}
Let $\mathcal{P}=\{P_n: 1\le n\le \omega\}$ be a   projection scheme on $X$.
We  say that $\mathcal{P}$  is  boundedly complete if for every  bounded sequence $(x_n)$  such that  $P_n(x_m)=x_n$ if $n\le m<\omega$, there is  a vector  $x\in P_\omega(X)$ with $P_n(x)=x_n$ for all $n\in\mathbb{N}$.

\end{definition}

\section{ $\mathcal P$-almost uniform convexity and asymptotic compactness of approximate asymptotic centers}

In this section,
we will  circumvent the possible emptiness of  asymptotic centers   by introducing the notion of $\mathcal{P}$-uniform convexity and more generally $\mathcal{P}$-almost uniform convexity. 

Firstly we extend the definition of uniform convexity when  the considered directions depend on a given operator.

\begin{definition}\label{kUC}
Let $A:X\to X$ be a  bounded operator. We say that $X$ is $A$-uniformly convex ($A$-UC) if for all $\epsilon>0$, there exists some $\gamma(A)>0$ such that for all $x,y\in B_X$  either
$$
\V A(x-y)\V<\epsilon \mbox{ or }\left\V {x+y\over 2}\right\V<1-\gamma(A).$$

If $\mathcal{A}=\{A_n: n\in\mathbb{N}\}$ is a family of linear bounded operators from $X$ into $X$, we say that $X$ is $\mathcal{A}$-UC if $X$ is $A_n$-UC for all $A_n\in\mathcal{A}$.

\end{definition} 
When $A=I$ the identity operator,   the $A$-UC definition simply reduces  to the definition of uniform convexity. However, our notion covers more general   structures  as the following example shows.

\begin{example}\label{puc}
Let $X=Y\oplus_2 Z$ where $Y$ is uniformly convex, $Z$ is any Banach space and the norm on $X$  is given by $\V x\V_X:=\left(\V y\V_Y^2+\V z\V_Z^2\right)^{1\over 2}$ if $x=y+z$ with $y\in Y$ and $z\in Z$.
Define $A:X\to X$ given by $A(x)=y$ if $x=y+z$. Then $X$ is $A$-UC.

\medskip

Indeed, fix some $\epsilon>0$ and $x_1,x_2\in B_X$ with $x_i=y_i+z_i$, $y_i\in Y$, $z_i\in Z$, $i=1,2$ and assume that $\V y_1-y_2\V\ge \epsilon$. Using that $Y$ is UC, there exists some $\delta_Y(\epsilon)$ such that \cite[Proposition 1, page 190]{Bea}:
$$
\left\V{y_1+y_2\over 2}\right\V_Y^2\le {\V y_1\V^2+\V y_2\V_Y^2\over 2}-\delta_Y(\epsilon).
$$
By the convexity of the function $\V \cdot\V^2$, when $\V\cdot\V$ is any norm, we have
$$
\begin{array}{lll}
\left\V {x_1+x_2\over 2}\right\V_X^2&\le& {\V y_1\V_Y^2+\V y_2\V_Y^2\over 2}-\delta_Y(\epsilon)+ {\V z_1\V_Z^2+\V z_2\V_Z^2\over 2}\\
  &= & {\V x_1\V^2_X+\V x_2\V_X^2\over 2}-\delta_Y(\epsilon).\\
  \end{array}
$$
Taking square root and adjusting the coefficients we obtain  that Definition \ref{kUC} is satisfied.
\end{example}

\medskip

In Section \ref{VLSsection}, we will be dealing with variable Lebesgue spaces for which finitely many atoms may appear in the extreme sets $p^{-1}(\{1\})$ and $p^{-1}(\{\infty\})$. In order to cover this case as well as Banach spaces endowed with a Schauder basis,  we extend Definition \ref{kUC} still preserving the relevant notions of convergence  that we expect to obtain, as it will be exposed in Theorem \ref{diam}.

\begin{definition}\label{AUC}
Let $X$ be a Banach space endowed with a projection scheme  $\mathcal{P}=\{P_n: 1\le n\le \omega\}$.  We say that $X$ is $\mathcal{P}$-almost uniformly convex ($\mathcal{P}$-AUC) if for every $n\in\mathbb{N}$, the projection $P_n$ can be split into the sum of two projections $P_n=A(P_n)+B(P_n)$,
 where $X$ is  $\mathcal{A}$-UC for $\mathcal{A}=\{A(P_n): n\in\mathbb{N}\}$  and rank of $B(P_n)$ is finite for every $n\in \mathbb{N}$.

\end{definition}

If $X$ has a Schauder basis, $X$ is trivially $\mathcal{P}$-AUC for the projection scheme associated to the canonical projections,  taking $A(P_n)=0$ for all $n\in\mathbb{N}$ (see Example \ref{Sch}).
Additionally, if $Z$ is finite dimensional in Example \ref{puc}, $X$ is $\mathcal{P}$-AUC, where $\mathcal{P}=\{I\}$,  the projection scheme just formed by the identity  operator,  since $I=A+B$ where $B$ is the identity on $Z$.

\medskip

We next introduce the notion of asymptotic compactness for a nonincreasing sequence of sets associated to a given  metrizable topology.

\begin{definition}\label{ac} Let $X$ be a Banach space endowed with a metrizable topo\-logy $\tau$. We say that a nonincreasing  sequence  $\{Z_s: s\in\mathbb{N}\}$ of closed bounded subsets in $X$ is $\tau$-asymptotically compact if for every sequence $(z_s)\subset X$ with $z_s\in Z_s$ for $s\in\mathbb{N}$,  there exists a $\tau$-convergent subsequence.

\end{definition}

When $\tau$ denotes the norm topology in a Banach space, the asymptotic compactness can be characterized by using the Kuratowski measure of noncompactness, which is defined for a bounded set $Z\subset X$ as
$$\alpha(Z):=\inf \{ \epsilon>0: Z  \mbox{ is covered
 by finitely many sets with diameter $\le \epsilon$}\}.$$
  It is well known that  $Z$ is relatively norm-compact if and only if $\alpha(Z)=0$ and it is not difficult to check that a nonincreasing  sequence  $\{Z_s: s\in\mathbb{N}\}$ of closed bounded subsets in $X$ is asymptotically compact for the norm topology  if and only if $\lim_s \alpha(Z_s)=0$.

\medskip

In case that $X$ is a Banach space equipped with a projection scheme $\mathcal{P}$, we will use the notation $\mathcal{P}$-asymptotic compactness when the given topology $\tau$ is the locally convex topology defined by the projections in $\mathcal{P}$ (see Section 3).
The following lemma provides us with an equivalent definition of $\mathcal{P}$-asymp\-totic compactness when $\mathcal{P}$ is a boundedly complete projection scheme. Somehow, it  resembles the one for the norm topology by using the Kuratowski measure of noncompactness.

\begin{lemma}\label{Pasymp}
Let $X$ be a Banach space endowed with a projection scheme $\mathcal{P}=\{P_n: n\in\mathbb{N}\cup\{\omega\}\}$ which is boundedly complete. 
 Assume that $(Z_s)$ is a sequence of closed bounded subsets of $X$ Then $(Z_s)$ is $\mathcal{P}$-asymptotically compact if and only if $\lim_s \alpha(P_n(Z_s))=0$ for all $n\in\mathbb{N}$.

\end{lemma}

\begin{proof}

When $(Z_s)$ is $\mathcal{P}$-asymptotically compact, it can  be easily checked that  $\lim_s\alpha(P_n(Z_s))=0$ para todo $n\in\mathbb{N}$ (without the requirement of being boundedly complete).

\medskip

Let us prove the converse: 
Let $(z_s)$ be a sequence in $X$, with $z_s\in Z_s$ for every $s\in\mathbb{N}$. Since $(I-P_\omega)$ has finite rank, we can assume, without loss of generality,  that there exists some $v_0\in (I-P_\omega)(X)$ such that $\lim_s (I-P_\omega)(z_s)=v_0$

Let $(\epsilon_k)$ be a nonnegative null sequence. Using a diagonal argument, we can construct a subsequence $(s_k)_k$ such that
$$
\alpha\left(P_k(Z_{s_k})\right)<\epsilon_k
$$
for all $k\in\mathbb{N}$. We prove that $\alpha(\{P_n(z_{s_k}):k\in\mathbb{N}\})=0$ for all $n\in\mathbb{N}$. Indeed, fix $n\in\mathbb{N}$. Take  $j>n$.
Then
$
\alpha(\{P_n(z_{s_k}):k\in\mathbb{N}\})=\alpha(\{P_n(z_{s_k}): s_k\ge s_j\})$. For $k,k^\prime$ with $s_k, s_{k^\prime}\ge s_j$ we have:
$$
\begin{array}{lll}
\V P_n(z_{s_k})-P_n(z_{s_{k^\prime}})\V & = & \V P_nP_j(z_{s_k})-P_nP_j(z_{s_{k^\prime}})\V\\
 & \le & {K}(\mathcal{P})\V P_j(z_{s_k})-P_j(z_{s_{k^\prime}})\V.
 \end{array}
$$
Therefore:
$$
\begin{array}{lll}
\alpha(\{P_n(z_{s_k}): s_k\ge s_j\})& \le& {K}(\mathcal{P})\alpha\left(\{P_{j}(z_{s_k}): s_k\ge s_j\}\right)\\
 & \le &  \mathcal{K}(\mathcal{P} )\alpha(P_j(Z_{s_{j}}))\le \epsilon_j.\\
\end{array}
$$
Consequently $\alpha(\{P_n(z_{s_k}):k\in\mathbb{N}\})=0$.  Taking a diagonal argument, we can assume, up to a subsequence,  that  $\{P_n(z_{s})\}_s$ is norm-convergent for all $n\in\mathbb{N}$. Then, there is a further sequence $(u_n)$ such that
$$
\lim_s\V P_n(z_s)-u_n\V=0\ \forall n\in\mathbb{N}.
$$
In particular $(u_n)$ is norm-bounded.
Take $n\le m<\omega$. We have
$$
P_n(u_m)=\lim_s P_nP_m(z_s)=\lim_s P_n(z_s)=u_n.
$$

Since $\mathcal{P}$ is boundedly complete, there exists $v_1\in P_\omega(X)$ with $P_n(v_1)=u_n$ for all $n\in\mathbb{N}$,   which implies  that $\lim_s P_n(z_s)=P_n(v_1)$ for all $n\in\mathbb{N}$. Since $v_1\in P_\omega(X)$ and $v_0\in (I-P_\omega)(X)$, the above implies that the sequence $(z_s)$ is $\mathcal{P}$-convergent to the vector $v_1+v_0\in X$ and the proof is complete.

\end{proof}

We aim to prove that approximate asymptotic centers provide us with families of sets which are $\mathcal{P}$-asymptotically compact.  The following lemma is technical:

\begin{lemma}\label{gamma}
Let $X$ be a Banach space. Let $A:X\to X$ be a linear bounded operator and assume that   $X$ is $A$-UC.
Let $C$ be a closed convex bounded subset and $(x_k)$ a sequence in $C$. Then
$$
\lim_{\delta\to 0^+}{\rm diam} \left[A(Z(\{x_k\},C,\delta))\right]=0.
$$

\end{lemma}
\begin{proof}

In case that $r(\{x_k)\},C)=0$, the sequence $(x_n)$ converges to some $x\in C$ and $\{x\}=Z(\{x_k\},C)$. Thus  $Z(\{x_k\},C,\delta)$ is contained in a ball centered at $x$ with radius $\delta$ and there is nothing to prove. Set $r:=r(\{x_k\},C)>0$.

Fix some $n\in\mathbb{N}$. 
Assume,  to the contrary, that there exists some $c>0$ such that:
$$
 0<c< \inf_{\delta>0}{\rm diam} \left[A(Z(\{x_k\},C,\delta))\right]=\lim_{\delta\to 0^+}{\rm diam}\left[A(Z(\{x_k\},C,\delta))\right]
$$
  For $\epsilon:=\frac{c}{ r+1}>0$ consider the corresponding  $\gamma(A)>0$ stated in the Definition \ref{kUC}. Find some $0<\delta<1$ such that $(r+\delta)(1-\gamma(A))<r$.

\medskip

Since $c<{\rm diam} \left[A(Z(\{x_k\},C,\delta/2))\right]$, we can find two points $y_1, y_2\in Z(\{x_k\},C,\delta/2)$ such that $\Vert  A(y_1-y_2)\Vert \geq c$.

 Then, there exists $k_0\in\mathbb{N}$ such that for all $k\ge k_0$:
 $$
 \V x_k-y_1\V \le r+\delta,\qquad \V x_k-y_2\V\le r+\delta.
 $$
For $k\ge k_0$, define $\displaystyle{u_k:=\frac{x_k-y_1}{ r+\delta}}$, $\displaystyle{v_k:=\frac{x_k-y_2}{ r+\delta}}$. Then, $\Vert u_k\Vert \le 1$, $\Vert v_k\Vert \le 1$ and $\V A(u_k-v_k)\V\ge \frac{c}{ r+\delta}>\frac{c}{ r+1}$. Thus,
$$
\left\V x_k-\frac{y_1+y_2}{ 2}\right\V=(r+\delta)\left\V \frac{u_k+v_k}{ 2}\right\V\le (r+\delta)(1-\gamma(A)),$$
for all $k\ge k_0$. Taking limit as $k$ goes to infinity:
$$
\limsup_k\left\V x_k-\frac{y_1+y_2}{ 2}\right\V\le (r+\delta)(1-\gamma(A))<r,
$$
which is a contradiction with the definition of $r$.
\end{proof}

Our main  result in this section is now the following theorem alongside with the subsequent corollary:

\begin{theorem}\label{diam}
Assume that $X$ is endowed with a projection scheme $\mathcal{P}=\{P_n:1\le n\le \omega\}$ which is boundedly complete and $X$ is $\mathcal{P}$-AUC.   Then, the sequence of subsets  $\{Z(\{x_n\},C,\delta_s):s\in\mathbb{N}\}$, where $(\delta_s)$ is a nonnegative null sequence, is $\mathcal{P}$-asymptotically compact.

\end{theorem}

\begin{proof}
Consider $\alpha(\cdot)$ the Kuratowski measure of noncompactness. 
Notice that $\alpha(\cdot)$ is monotone and sublinear in the sense that
 $\alpha(Z_1)\le \alpha(Z_2)$ if $Z_1\subset Z_2$ and $\alpha(Z_1+Z_2)\le \alpha(Z_1)+\alpha(Z_2)$ for all sets $Z_1,Z_2$ in $X$. Additionally
$\alpha(Z)\le {\rm diam}(Z)$ for all $Z\subset X$.

Fix some $n\in\mathbb{N}$.
According to Definition \ref{kUC}, we can assume that  $P_n=A_n+B_n$  where $B_n$ is a finite-rank operator and $X$ is $\mathcal{A}$-UC for $\mathcal{A}=\{A_n: n\in \mathbb{N}\}$.
Notice that for all bounded set $Z\subset X$ we have that $P_n(Z)\subset A_n(Z)+B_n(Z)$ and $B_n(Z)$ is relatively compact. Hence, it suffices to show that

 $$\lim_{s}{\rm diam} \left(A_n(Z(\{x_n\},C,\delta_s)\right)=0$$
for all $n\in\mathbb{N}$ and this follows from Lemma \ref{gamma}.

\end{proof}

\begin{corollary}\label{afps}
Let $X$ be endowed with a boundedly complete projection scheme $\mathcal{P}$ for which $X$ is $\mathcal{P}$-AUC.
 Let $C$ be a closed convex bounded subset, and $T:C\to C$ a nonexpansive mapping. Then, there exists an a.f.p.s.  $(y_n)$ in $C$  which is $\mathcal{P}$-convergent.

\end{corollary}

\begin{proof}
Take $(x_k)$ an a.f.p.s. in $C$.
It is enough to pick $y_s\in Z_s:=Z(\{x_k\},C,\delta_s)$ such that $\V y_s-T(y_s)\V <{1\over s} $ where $(\delta_s)$ is a decreasing sequence converging to 0,  apply Theorem \ref{diam} and the definition of $\mathcal{P}$-asymptotic compactness.

\end{proof}

It is well known that the norm-limit of an a.f.p.s. (when it does exist) is a fixed point of a nonexpansive mapping. 
In particular, this norm convergence may occur  whenever  $\mathcal{P}$-convergence and norm-convergence are identical for bounded sequences. Following the standard Banach terminology, we can introduce the following definition:
\begin{definition}
 If $X$ is a Banach space equipped with a projection scheme $\mathcal{P}$,  we say  that  $X$ has the $\mathcal{P}$-Schur property if every $\mathcal P$-convergent bounded sequence converges in norm.
\end{definition}

As a first application we can deduce the following:

\begin{corollary}\label{Schur} Let $X$ be a Banach space endowed with a projection scheme $\mathcal{P}$ which is $\mathcal{P}$-AUC and boundedly complete. If $X$ has the $\mathcal{P}$-Schur property, then $X$ has the FPP.
\end{corollary}

\begin{proof}
Let $C$ be a closed convex bounded subset, $T:C\to C$ nonexpansive and $(x_n)$ an afps in $C$ which is $\mathcal{P}$-convergent. If the $\mathcal{P}$-Schur property holds, then the norm limit of $(x_n)$ is a fixed point.
\end{proof}

\section{Fixed point property beyond reflexivity: main results}

Throughout  this section $X$ will denote a Banach space equipped  with a projection scheme $\mathcal{P}$ which is boundedly complete.  From Corollary \ref{Schur} we can assume that   $X$ does not have  the $\mathcal{P}$-Schur property and we can always find   normalized sequences which are $\mathcal{P}$-convergent to the null vector.  Thus the following set
$$
\mathcal{S}_1(X):=\{(x_n)\subset X: \mathcal{P}-\lim_n x_n=0, \limsup_n\V x_n\V=1\}
$$
is nonempty.

\medskip

Next,  we are going to deal with norms  that  share some of the  features
of the  $\ell_1$-asymptotic triangle equality.  It is a well-known fact that 
if $(x_n)$ is a sequence in $\ell_1$  converging coordinatewise to 0,  for every   $x\in\ell_1$ one has 
\begin{equation}\label{l1}
\limsup_n \V x_n+x\V=\limsup_n\V x_n\V+\V x\V.
\end{equation}
In other words, the triangle inequality becomes asymptotically sharp  for the norm of the sum of the vectors $x_n$ and $x$.

In what follows, we will substitute coordinatewise convergence in $\ell_1$  by $\mathcal{P}$-convergence on a Banach space equipped with a projection scheme, and the vector $x$ in (\ref{l1}) will be required to satisfy a suitable assumption.
Our motivation comes from properties of Schauder bases that have been used to establish the existence of fixed points for nonexpansive mappings in \cite{todos}.

\medskip

\begin{definition}
Given a Banach space with  a projection scheme $\mathcal{P}=\{P_k:k\in\mathbb{N}\cup\{\omega\}\}$,  for all $k\in\mathbb{N}$ we
define
$$S(P_k):=\inf\{\limsup_n\V x_n+x\V: ((x_n),x)\in \mathcal{S}_1(X)\times S_X, P_k(x)=0\}.
$$
We say that $X$ is of  {\it $\mathcal{P}$-asymptotically  $\ell_1$-type} if
$$
S(\mathcal{P}):=\lim_k S(P_k)=2.
$$

\end{definition}

Notice that the above limit is well-defined since the sequence $(S(P_k))_k$ is nondecreasing (because $Ker(P_{k+1})\subset Ker(P_k)$) and  $S(P_k)\le 2$ for all $k\in\mathbb{N}$. When $X=\ell_1$ is endowed with the projection scheme determined by its canonical  Schauder basis,   $S(P_k)=2$ for all $k$ from (\ref{l1}).  This explains the terminology  {\it asymptotically $\ell_1$-type}.

\medskip

The coefficient $S(P_k)$ measures the sharpness of the triangle inequality for
$\limsup_n\V x_n+x\V$ whenever $P_k(x)=0$ and $\mathcal{P}-\lim_n x_n=0$. This is better understood  after the next technical lemma.

\begin{lemma} Given $k\in \mathbb{N}$, for every bounded sequence $(x_n)$ in $X$  $\mathcal{P}$-convergent to 0,  and for every $x\in X$ with $P_k(x)=0$, we have:
$$\limsup_n\Vert x_n+x\Vert \geq (S(P_k)-1)(\limsup_n\Vert x_n\Vert +\Vert x\Vert).$$
\end{lemma}
\begin{proof}
 If $x=0$ or $\lim_n\V x_n\V=0$, the assertion is immediate. 
 
 Thus, we can assume $x\ne 0$ and  $l:=\limsup_n\V x_n\V>0$. Take a subsequence $(x_{n_s})$ such that $l:=\lim_s\V x_{n_s}\V$ and, without loss of generality, we can assume that $\V x_{n_s}\V\ne 0$ for all $s\in\mathbb{N}$, $\mathcal{P}-\lim_s {x_{n_s}\over \V x_{n_s}\V}=0$ and, given $\epsilon>0$ that

$$ \left\Vert \frac{x_{n_s}}{\Vert x_{n_s}\Vert }+\frac{x}{\Vert x\Vert}\right\Vert>S(P_k)-\epsilon.$$

Take $f_s\in X^*$ with $\V f_s\V=1$, such that
$$ \left\Vert \frac{x_{n_s}}{\Vert x_{n_s}\Vert }+\frac{x}{\Vert x\Vert}\right\Vert=f_s\left(\frac{x_{n_s}}{\Vert x_{n_s}\Vert }+\frac{x}{\Vert x\Vert}\right)=f_s\left(\frac{x_{n_s}}{\Vert x_{n_s}\Vert }\right)+f_s\left(\frac{x}{\Vert x\Vert}\right).$$
Then, $$f_s(x_{n_s}/\Vert x_{n_s}\Vert)\geq S(P_k)-\epsilon-1\quad {\rm and} \quad  f_s(x/\Vert x\Vert)\geq S(P_k)-\epsilon-1.$$

For all $s\in\mathbb{N}$, we have
$$
\begin{array}{lll}
\V x_{n_s}+x\V & \ge & f_s(x_{n_s}+x)
 =\V x_{n_s}\V f_s\left(\frac{x_{n_s}}{ \V x_{n_s}\V}\right)+\V x\V f_s\left(\frac{x}{ \V x\V}\right)\\
 &\ge & (S(P_k)-\epsilon-1)\left[\V x_{n_s}\V+\V x\V\right].\\
\end{array}
$$
Taking upper  limit when $s$ goes to infinity, the arbitrariness  of $\epsilon$ implies that
$$
\limsup_s\V x_{n_s}+x\V\ge (S(P_k)-1)\left[\lim_s\V x_{n_s}\V +\V x\V\right]= (S(P_k)-1)\left[\limsup_n\V x_{n}\V +\V x\V\right].
$$
Finally,
$$
\limsup_n\V x_n+x\V\ge \limsup_s\V x_{n_s}+x\V\ge (S(P_k)-1)\left[\limsup_n\V x_{n}\V +\V x\V\right]
$$
as we wanted to prove.

\end{proof}

The following lemma is once more technical and will be used for the proof of Theorem \ref{main}.

\begin{lemma} \label{bound} Let $X$ be a Banach space endowed with a Projection Scheme
$\mathcal{P}=\{P_k:1\le k\le \omega\}$. Fix some $k\in\mathbb{N}$. Then for every pair of  bounded sequences $(x_n), (y_n)\subset X$ with $\mathcal{P}-\lim_n x_n=x$, $\mathcal{P}-\lim_n y_n=y$, the following inequality holds:
\begin{equation}\label{ka}
\begin{array}{lll}
\displaystyle{\V I-P_k\V\limsup_m \limsup_n\V x_n-y_m\V} &\ge& (S(P_k)-1)  \limsup_n\V x_n-x\V\\
                 & + & (S(P_k)-1)^2 \limsup_m\V y_m-y\V \\
                  &+ & (S(P_k)-1)^2 \V (I-P_k)(x-y)\V\\
\end{array}
\end{equation}

In particular, if $\mathcal{P}$ is premonotone and $X$ is of $\mathcal{P}$-asymptotically $\ell_1$-type, we have
\begin{equation}\label{dl}
\limsup_m \limsup_n\V x_n-y_m\V\ge \limsup_n\V x_n-x\V+\limsup_n\V y_n-y\V.
\end{equation}

\end{lemma}
\begin{proof}
Notice that $\limsup_n\V x_n-x\V=\limsup_n\V (I-P_k)(x_n-x)\V$, $\limsup_n\V y_n-y\V=\limsup_n\V (I-P_k)(y_n-y)\V$ and $\mathcal{P}-\lim_n (I-P_k)(x_n-x)=0=\mathcal{P}-\lim_n (I-P_k)(y_n-y)$.

Then, for all $m\in\mathbb{N}$,
$$
\begin{array}{ll}
 &\displaystyle{\V I-P_k\V \limsup_n\V x_n-y_m\V} \\
 \ge &\displaystyle{ \limsup_n\V (I-P_k)(x_n-x)+(I-P_k)(x-y_m)\V}\\
  \ge &\displaystyle{ (S(P_k)-1) \left[ \limsup_n\V (I-P_k)(x_n-x)\V+\V(I-P_k) (x-y_m)\V\right]}\\
  = &\displaystyle{ (S(P_k)-1) \left[ \limsup_n\V x_n-x\V+\V(I-P_k) (x-y_m)\V\right]}
\end{array}
$$
Likewise, taking limit when $m$ goes to infinity:
$$
\begin{array}{ll}
& \displaystyle{\limsup_m\V (I-P_k)(y_m-x)\V }\\
= &\displaystyle{ \limsup_m\V (I-P_k)(y_m-y)+ (I-P_k)(y-x)\V
}\\
 \ge & \displaystyle{ (S(P_k)-1)\limsup_n\V y_n-y\V+(S(P_k)-1)\V (I-P_k)(y-x)\V.}
 \end{array}
$$

Combining the above inequalities we obtain (\ref{ka}). Finally, letting $k\to\infty$, we obtain (\ref{dl}).

\end{proof}

Next, we introduce a quantity that measures the behavior of
$\limsup_n \V x_n+\lambda x\V$, as $\lambda\to 0^+$
for sequences $(x_n)$ satisfying $\mathcal{P}$-$\lim_n x_n=0$ and vectors $x$ fixed by $P_k$. More precisely, for every $k\in\mathbb{N}$, we define
  $\Delta_{P_k}:[0,1]\to [0,+\infty)$ given by
$$
\Delta_{P_k}(\lambda):=\sup\{\limsup_n \V x_n+\lambda x\V-1: ((x_n),x)\in \mathcal{S}_1(X)\times S_X; P_k(x)=x\}
$$
and
$$
\Delta(\mathcal{P}):=\sup_{k\in\mathbb{N}} \left\{\limsup_{\lambda\to 0^+}{\Delta_{P_k}(\lambda)\over \lambda}\right\}.
$$
From the triangular inequality, the upper estimations $\Delta_{P_k}(\lambda)\le \lambda$ and $\Delta(\mathcal{P})\le 1$ always hold. 
The upper bound $\Delta(\mathcal{P})=1$ is attained when
$X=(\ell_1,|\cdot|_1)$ and $(P_k)$ is the sequence of canonical projections associated with its Schauder basis. Since $\ell_1$ fails the FPP, we will be interested in Banach spaces for which $\Delta(\mathcal{P})$ is far away from $1$.

\medskip

Some general background before  stating the main result:

Assume that $C$ is a closed convex bounded subset of $X$ and $T:C\to C$ is a fixed-point-free mapping. Standard arguments show that there exist some $a>0$ and  a further closed, convex,  $T$-invariant subset $C_0\subset C$ with ${\rm diam}(H)> a$ for all $H\subset C_0$ closed, convex and $T$-invariant (see \cite{Lin-2008} or \cite[Section 2]{japon2}).
In particular, for all $\mathcal{P}$-convergent $(x_n)\subset C_0$ a.f.p.s. with $\mathcal{P}-\lim_n x_n=x_0$,  we have $\limsup_n\V x_n-x_0\V\ge {a\over 4}$ since, otherwise, the set $\{z\in C_0: \limsup_n\V x_n-z\V\le {a\over 2}\}$ would be a nonempty, closed, convex, $T$-invariant subset of $C_0$ and its diameter would be less than  of equal to $a$. Summarizing, whenever $T$ has no fixed point, we may always assume that $T:C\to C$ verifies that $\inf\{\limsup_n\V x_n-x_0\V: (x_n) \subset C \mbox{ is an a.f.p.s. and } \mathcal{P}-\lim_n x_n=x_0\}$ is strictly greater than $0$.

\medskip

Moreover, summarizing the results developed in the previous section, we can claim the following:
 
 \medskip
 
 Let $X$ be a Banach space which can be equipped with a projection scheme $\mathcal{P}$ which is boundedly complete and  $\mathcal{P}$-AUC. Let  $C$ be a  closed, convex, bounded subset of $X$ and $T:C\to C$ a nonexpansive mapping.  We  denote:
$$
\mathcal{A}(C):=\{((x_n)_n,x_0)): (x_n)\subset C, (x_n) \mbox{ is an a.f.p.s.  and } \mathcal{P}-\lim_n x_n=x_0\}.
$$
Then  $\mathcal{A}(C)\not=\emptyset$ by Corollary \ref{afps}.

We are now in a position to state and prove the major outcome of this paper.

\begin{theorem} \label{main}
Let $X$ be a Banach space endowed with a premonotone boundedly complete projection scheme $\mathcal{P}:=\{P_n: 1\le n\le \omega\} $ such that

\begin{enumerate}

\item $X$ is $\mathcal{P}$-AUC

\item  $X$ is of $\mathcal{P}$-asymptotically $\ell_1$-type and

\item  $
\Delta(\mathcal{P})<{1\over 1+2{K}(\mathcal{P})}.
$
\end{enumerate}

Then $X$ has the FPP.

\end{theorem}

\begin{proof}


\medskip

If $T$ has a fixed point, there is nothing to prove. Otherwise,  we can assume that
 $$
 \rho(C):=\inf\{\limsup_n\V x_n-x_0\V: \{(x_n)_n,x_0\}\in \mathcal{A}(C)\}>0,
 $$
After rescaling the norm, we may assume that $\rho(C)=1$.
\medskip

Take $s>0$ such that $\Delta(\mathcal{P})< s < {1\over 1+2K(\mathcal{P})}$.

\medskip

Define the function $a(\epsilon):=(1+2\epsilon)+K(\mathcal{P})2(1+\epsilon)$ for all $\epsilon>0$. Since $\lim_{\epsilon\to 0^+} a(\epsilon)s=[1+2K(\mathcal{P})]s<1$ we can fix some
 $\epsilon>0$  with
\begin{equation}\label{lambda}
a(\epsilon)s+(1+5\epsilon)<2.
\end{equation}

\medskip

Fix $\{(x_n),x_0\}\in\mathcal{A}(C)$ with
$$\limsup_n\V x_n-x_0\V< 1+\epsilon.$$

The above inequality
 implies that $\limsup_m\limsup_n\V x_n-x_m\V<2(1+\epsilon)$. Without loss of generality, we can assume that $\limsup_n\V x_n-x_m\V<2(1+\epsilon)$ for all $m\in\mathbb{N}$.
Set
$$
C_0:=\{z\in C: \limsup_n\V x_n-z\V\le 2(1+\epsilon)\},
$$
and we can assume that $(x_n)\subset C_0$. The set $C_0$ is a $T$-invariant, closed, convex, nonempty subset of $C$. Denote
$$
\rho(C_0):=\inf\{\limsup_n\V y_n-y_0\V: \{(y_n),y_0\}\in \mathcal{A}(C_0)\}\ge 1.
$$

 We will prove that there exists some $\{(y_n),y_0\}\in \mathcal{A}(C_0)$ such that
$$
r(\{y_n\},C_0)<2\rho(C_0)
$$
and this  will eventually lead to a contradiction.

\medskip

Since later we shall choose $k_0\in\mathbb{N}$ sufficiently large to satisfy some required conditions,   $X$ is of $\mathcal{P}$-asymptotically $\ell_1$-type,  we may assume without loss of generality that $S(P_k)>1$ for all $k\in\mathbb{N}$.

\medskip

 Take any  $\{(y_n),y_0\}\in\mathcal{A}(C_0)$. For $k\in\mathbb{N}$, we will  separately estimate the \lq\lq head''  $\V P_k(x_0-y_0)\V$ and the \lq\lq tail'' $\V (I-P_k)(x_0-y_0)\V$
 of the vector $ x_0-y_0$.

\medskip

First:

\begin{equation}\label{pk}
\begin{array}{rrl}
\V P_{k}(x_0-y_0)\V&=& \limsup_m\V P_{k}(x_0-y_m+y_m-y_0)\V\\
   & = &\limsup_m\V P_{k}(x_0-y_m)\V\\
 &= &\limsup_m\limsup_n \V P_{k}(x_0-x_n+x_n-y_m)\V\\
  &=&\limsup_m\limsup_n\V P_{k}(x_n-y_m)\V\\
  &\le & \V P_k\V \limsup_m\limsup_n\V x_n-y_m\V\le K(\mathcal{P})2(1+\epsilon).\\
\end{array}
\end{equation}

Next, using Lemma \ref{bound} we have:

\begin{equation}
\begin{array}{ll}
 & \displaystyle{ \V (I-P_k)(x_0-y_0)\V+\limsup_n\V y_n-y_0\V}\\
\le  &\displaystyle{ {1\over (S(P_k)-1)^2}\limsup_m\limsup_n\V x_n-y_m\V  }
 - \displaystyle{{1\over S(P_k)-1} \limsup_n\V x_n-x_0\V}   \\
 \le &\displaystyle{ {1\over (S(P_k)-1)^2}2(1+\epsilon) - {1\over S(P_k)-1} }.  \end{array}
\end{equation}
Thus,

\begin{equation}\label{Q}
\V (I-P_k)(x_0-y_0)\V+\limsup_n\V y_n-y_0\V\leq A_k:= \displaystyle{ {1\over (S(P_k)-1)^2}2(1+\epsilon) - {1\over S(P_k)-1} }.
\end{equation}
and $\lim_k A_k=1+2\epsilon$.

\medskip

Since $\mathcal{P}-\lim_n x_n=x_0$,  we have $\lim_n\V (I-P_\omega)(x_n-x_0)\V=0$ and consequently:
$$\limsup_n\V P_\omega(x_n-x_0)\V=\limsup_n\V x_n-x_0\V<1+\epsilon.$$
Fix $N_0\in\mathbb{N}$ with $x_{N_0}\in C_0$ and
$$
\V x_{N_0}-x_0\V <1+\epsilon;\ \V P_\omega(x_{N_0}-x_0)\V<1+\epsilon;\ \V (I-P_\omega)(x_{N_0}-x_0)\V<\epsilon.
$$

For every $k\in\mathbb{N}$ we have:
$$
\begin{array}{lll}
\V P_k(x_{N_0}-x_0)\V&\le &\V (P_k-P_\omega)(x_{N_0}-x_0)\V+ \V P_\omega(x_{N_0}-x_0)\V \\
 &\le & \V (P_k-P_\omega)(x_{N_0}-x_0)\V+1+\epsilon;
                                    \end{array}
$$
and
$$
\begin{array}{lll}
\V (I-P_{k})(x_{N_0}-x_0)\V&\le & \V (P_\omega-P_{k})(x_{N_0}-x_0)\V+\V (I-P_\omega)(x_{N_0}-x_0)\V\\
             &\le &\V (P_\omega-P_{k})(x_{N_0}-x_0)\V+\epsilon\\
             \end{array}.
$$

By the definition of a projection scheme (see Definition \ref{PS}.b)), the sequence $(P_k)_k$ converges pointwise to $P_\omega$.
Thus, we can take $k_0$ large enough so that $$\V (P_\omega-P_{k_0})(x_{N_0}-x_0)\V<\epsilon\quad  {\rm and} \quad A_{k_0}<1+3\epsilon.$$

Consequently, we deduce:
\begin{equation}\label{k}
\V P_{k_0}(x_{N_0}-x_0)\V<1+2\epsilon; \
\V (I-P_{k_0})(x_{N_0}-x_0)\V < 2\epsilon;
\end{equation}

\medskip

Then, for all $\{(y_n),y_0\}\in\mathcal{A}(C_0)$, we obtain:

\begin{equation}
\begin{array}{lll}
\V P_{k_0}(x_{N_0}-y_0)\V& \le & \V P_{k_0}(x_{N_0}-x_0)\V+ \V P_{k_0}(x_0-y_0)\V\\
    &\le & (1+2\epsilon)+K(\mathcal{P})2(1+\epsilon)=a(\epsilon),\\
    \end{array}
\end{equation}

and

\begin{equation}
\begin{array}{lll}\label{imp}
\V (I-P_{k_0})(x_{N_0}-y_0)\V +\limsup_n\V y_n-y_0\V &\le& \V (I-P_{k_0})(x_{N_0}-x_0)\V\\
&+ & \V (I-P_{k_0})(x_0-y_0)\V\\
&+&\limsup_n\V y_n-y_0\V\\
&\le & 2\epsilon+1+3\epsilon=1+5\epsilon.\\
\end{array}
\end{equation}

\medskip

Once  $x_{N_0}$ and $P_{k_0}$ have been fixed, the assumptions imply $\limsup_{\beta\to 0^+}{\Delta_{P_{k_0}(\beta)}\over \beta}\le \Delta(\mathcal{P})<s $ and there exists some $\beta_0\in(0,1)$ such that
$$
\Delta_{P_{k_0}}(\beta) <\beta s
$$
for every $\beta\le \beta_0$ and
\begin{equation}\label{beta}
\limsup_n\V x_n+\beta x\V\le 1+\beta s
\end{equation}
 for every $\{(x_n),x\}\in \mathcal{S}_1(X)\times S_X$ satisfying $P_k(x)=x$ and $\beta\in (0, \beta_0]$.

\medskip

Take $\lambda_0\in (0,1)$ such that $\lambda_0<{\beta_0\over a(\epsilon)}$.

\medskip

In particular, multiplying by $\lambda_0$ in (\ref{lambda}) and summing $2\rho(C_0)$ on each side of the equation we have:
$$
2\rho(C_0)+\lambda_0 a(\epsilon) s+\lambda_0(1+5\epsilon)<2\lambda_0 +2\rho(C_0) \le2\lambda_0 \rho(C_0)+2\rho(C_0).
$$
since $\rho(C_0)\ge 1$.
The above implies that:
$$
(2-2\lambda_0)\rho(C_0)+\lambda_0a(\epsilon)s+\lambda_0(1+5\epsilon)<2\rho(C_0).
$$
Take $t_0>0$ small enough so that
\begin{equation}\label{t0}
(2-2\lambda_0)(\rho(C_0)+t_0)+\lambda_0a(\epsilon)s+\lambda_0(1+5\epsilon)<2\rho(C_0).
\end{equation}

\medskip

Finally, take $((y_n),y_0)\in\mathcal{A}(C_0)$ with $r:=\limsup_n\V y_n-y_0\V\le \rho(C_0) + t_0$.

\medskip

  We have:
$$
\begin{array}{ll}
&\limsup_n\left\V y_n-y_0+\lambda_0 P_{k_0} (x_{N_0}-y_0)\right\V   \\
=& \displaystyle{r \limsup_n\left\V {y_n-y_0\over r}+{\lambda_0 \V P_{k_0} (x_{N_0}-y_0)\V \over r} {P_{k_0} (x_{N_0}-y_0)\over \V P_{k_0} (x_{N_0}-y_0)\V }\right\V}. \\
\end{array}
$$

\medskip

Notice that  $1\le r$ and  $\displaystyle{{\lambda_0 \V P_{k_0} (x_{N_0}-y_0)\V \over r}\le \lambda_0 a(\epsilon)<\beta_0.}$ Using (\ref{beta}) we deduce:

\begin{equation}\label{cota}
\begin{array}{lll}
\limsup_n\V y_n-y_0+\lambda_0 P_{k_0} (x_{N_0}-y_0)\V& \le & r\left[  1+{\lambda_0 \V P_{k_0} (x_{N_0}-y_0)\V\over r} s \right]\\
&= &\limsup_n\V y_n-y\V+\lambda_0  \V P_{k_0} (x_{N_0}-y_0)\V s\\
&\le & \limsup_n\V y_n-y\V+\lambda_0 a(\epsilon) s.\\
\end{array}
\end{equation}

\medskip

For all $n,m\in\mathbb{N}$,  we have:
\begin{equation}\label{cota2}
\begin{array}{lll}
\V y_n-(1-\lambda_0) y_m-\lambda_0 x_{N_0}\V&=&\V (y_n-y_0)-(1-\lambda_0) (y_m-y_0)-\lambda_0 (x_{N_0}-y_0)\V\\
& \le &
(1-\lambda_0)\V y_m-y_0\V +\V y_n-y_0-\lambda_0 P_{k_0}(x_{N_0}-y_0)\V\\
& +&\lambda_0 \V (I-P_{k_0})(x_{N_0}-y_0)\V.
\end{array}
\end{equation}

Taking limit as $n$ goes to infinity and using (\ref{cota}) and (\ref{cota2}):
$$
\begin{array}{ll}
& \limsup_n\V y_n-(1-\lambda_0) y_m-\lambda_0 x_{N_0}\V\\
\le & (1-\lambda_0)\V y_m-y_0\V+ \limsup_n\V y_n-y_0\V +\lambda_0 a(\epsilon) s+\lambda_0  \V (I-P_{k_0})(x_{N_0}-y_0)\V
\end{array}
$$

Taking limit when $m$ goes to infinity and using (\ref{imp}):
$$
\begin{array}{ll}
&\limsup_m \limsup_n\V y_n-(1-\lambda_0) y_m-\lambda_0 x_{N_0}\V\\
\le & (1-\lambda_0)\limsup_m\V y_m-y_0\V+ \limsup_n\V y_n-y_0\V+ \lambda_0 a(\epsilon) s\\
  +&\lambda_0 (1+5\epsilon-\limsup_n\V y_n-y_0\V).\\
\end{array}
$$

Consequently

$$
\begin{array}{lll}
& & \limsup_m\limsup_n\V y_n-(1-\lambda_0) y_m-\lambda_0 x_{N_0}\V\\
  & \le &   (2-2\lambda_0)\limsup_n\V y_n-y_0\V+   \lambda_0 a(\epsilon) s  +\lambda_0 (1+5\epsilon) \\
  &\le & (2-2\lambda_0)(\rho(C_0)+t_0)+  \lambda_0 a(\epsilon) s  +\lambda_0 (1+5\epsilon)\\
  &<&2\rho(C_0);\\
\end{array}
$$
where we have used (\ref{t0}).

\medskip

Choose $\alpha<2\rho(C_0)$  such that
$$\limsup_m\limsup_n\V y_n-(1-\lambda_0) y_m-\lambda_0 x_{N_0}\V<\alpha$$
and  $m\in\mathbb{N}$ with
$$
\limsup_n\V y_n-(1-\lambda_0) y_m-\lambda_0 x_{N_0}\V\le \alpha<2\rho(C_0).
$$
Since $(1-\lambda_0) y_m+\lambda_0 x_{N_0}\in C_0$ by convexity, we deduce that the set
$$
C_1:=\{ z\in C_0: \limsup_n\V y_n-z\V\le \alpha\}
$$
is closed, convex, $T$-invariant and nonempty. Take $\{(z_n), z_0\}\in \mathcal{A}(C_1)$. Using again inequality (\ref{dl}) we have:
$$
\begin{array}{lll}
2\rho(C_0)&>& \alpha\ge \limsup_m\limsup_n\V y_n-z_m\V\\
   & \ge & \limsup_n\V y_n-y_0\V+\limsup_m\V z_m-z_0\V\ge 2\rho(C_0).
\end{array}
$$

The above contradiction implies  that $T$ has a fixed point, as we wanted to prove.
\end{proof}

\def\ve{\vert\vert\vert}

 We first present an immediate application of Theorem \ref{main}, which shows that the class  of nonreflexive Banach spaces with the FPP is substantially larger than previously known.  It is worth recalling that the question of whether the FPP implies reflexivity remained open for more than forty years and the question of whether the FPP is preserved by direct 2-sums is still open.

\begin{theorem}\label{sum} Let $X=(\ell_1,\ve\cdot \ve)\oplus_2 \ell_2$, where the norm $\ve \cdot\ve$ is given by (\ref{Linnorm}). Then $X$ has the FPP.

\end{theorem}

\begin{proof}

We fix some notation. For $x\in X$ we write $x=x^1+x^2$ where $x^1\in\ell_1$ and $x^2\in \ell_2$. Let $(R_n)_n$ be the  projections on $\ell_1$ associated to its canonical Schauder basis.
 We define a projection scheme on $X$ in the following way: we set $P_\omega=I$, the identity on X, and for all $n\in\mathbb{N}$ we define the projection:  $P_n(x)=P_n(x^1+x^2):=R_n(x^1)+x^2$. Then $\mathcal{P}:=\{P_n: 1\le n\le \omega\}$ is a projection scheme which is boundedly complete.

 We prove that $\mathcal{P}$ satisfies the remaining assumptions of Theorem \ref{main}.

Observe that $P_n=R_n+I_2$  where $I_2(x^1+x^2)=x^2$ and every $R_n$ has finite rank. Furthermore, $X$ is $\mathcal{A}$-UC where $\mathcal{A}$  consists solely of the projection $I_2$, since the direct sum is endowed with the $\ell_2$-norm (see Example \ref{puc}).

\medskip

Additionally,  $X$ is of $\mathcal{P}$-asymptotically $\ell_1$-type:  Indeed,  let $(x_n)$ be a $\mathcal{P}$-null sequence with $\limsup_n\V x_n\V=1$. In particular $\lim_n R_k(x_n^1)=0$ for all $k\in\mathbb{N}$ and $\lim_n I_2(x_n^2)=x_n^2=0$. Consequently, $(x_n^1)$ is  coordinatewise convergent to 0  and we can assume, taking a subsequence and  without loss of generality, that $(x_n^1)$ is a block basic sequence.

 Take $x=x^1+x^2\in S_X$ with $P_k(x^1+x^2)=R_k(x^1) +I_2(x^2)=0$. Thus  $x^2=0$ and
 $$
\begin{array}{lll}
\limsup_n\V x_n+x\V&=&\limsup_n\V (I-P_k)(x_n)+x\V\\
                &=&\limsup_n\ve (I-R_k)(x_n^1)+x^1\ve\\
                                  & \ge & \gamma_k \left(\limsup_n\ve  (I-R_k)(x_n^1)\ve+ \ve x^1\ve\right)\\
                                  & = & \gamma_k \left(\limsup_n\V  x_n\V + \V x\V\right)=2\gamma_k\\
                                  \end{array},
$$
which shows that $S(P_k)\ge 2\gamma_k$ and  $X$ is of $\mathcal{P}$-asymptotically $\ell_1$ type (see \cite{todos}).

\medskip

We next prove that $\limsup_{\lambda\to 0^+}{\Delta_{P_k}(\lambda)\over \lambda}=0$ for all $k\in\mathbb{N}$, which implies that $\Delta(\mathcal{P})=0$.

 Fix $k\in\mathbb{N}$ and let $x=x^1+x^2\in X$ with $P_k(x)=x$ and $\V x\V=1$. Consider $(x_n)$ a $\mathcal{P}$-null convergent sequence with $\limsup_n\V x_n\V=1$. Without loss of generality, we can assume that $x_n=x_n^1$ with $(x_n^1)$ a weak$^*$-convergent sequence in $\ell_1$ and $\limsup_n \ve x_n^1\ve =1$. Let $0<\lambda<{1\over \gamma_k}-1$. Then
$$
\begin{array}{lll}
\limsup_n\V x_n+\lambda x\V^2&= &\lambda^2\V x^2\V^2+\limsup_n\ve x_n^1+\lambda x^1\ve^2\\
& =& \lambda^2\V x^2\V^2+ \limsup_n \ve x_n^1\ve^2\le   \lambda^2+1,\\
\end{array}
$$
 which implies $\limsup_n\V x_n+\lambda x\V\le  (\lambda^2+1)^{1\over 2}$ and $\Delta_{P_k}(\lambda)\le (\lambda^2+1)^{1\over 2}-1$ for $0<\lambda <{1\over \gamma_k}-1$.
Therefore
$$
\limsup_{\lambda\to 0^+}{\Delta_{P_k}(\lambda)\over \lambda}=0.
$$
All the assumptions of Theorem  \ref{main} are satisfied and we can conclude that $X$ has the FPP.

\end{proof}

Since the preceding argument applies to every uniformly convex Banach space we can conclude:

 \begin{corollary}\label{uciso} Every uniformly convex Banach space $Y$ can be isometrically embedded in a nonreflexive Banach space with the FPP.

  \end{corollary}
\begin{proof}
Define
$X:=(\ell_1,\ve\cdot \ve)\oplus_2 Y$. Using the same arguments as in Theorem \ref{sum}, we can prove the FPP from Theorem \ref{main}.

\end{proof}

Our main theorem also includes, as a particular case, the class of Banach spaces endowed with a boundedly complete Schauder basis studied in \cite{todos}.
The following notion is introduced in  \cite[Definition 2.2] {todos}:

\begin{definition} \label{ss} Let $X$ be a Banach space with a  Schauder basis and denote $(P_k)$ the standard projections. The norm of $X$ is called near-infinity concentrated (n.i.c.) if it has the following properties:

\begin{itemize}

\item[(1)]
It is a sequentially separating norm.

\item[(2)] It is a premonotone norm.

\item[(3)]  There exist $R_0>5$ and $M\in [0,1)$ such that for every $k\in \mathbb{N}$, there exists a function $F_k:(0,\infty) \to [0, \infty)$ satisfying the following conditions:
\begin{itemize}

\item[(a)] $\lim_{\lambda\to 0^+} F_k(\lambda) \leq M/R_0$.

\item[(b)]

For every bounded pointwise-null sequence $(x_n)$ with $\liminf_n \V x_n \V  \geq 1$, for all $\lambda\in (0,\infty)$, and for every $z\in X$ with  $P_k(z)=z$ and $\V z\V \leq R_0$, we have
$$\limsup_n \V x_n+\lambda z\Vert \leq \limsup \Vert x_n\Vert +F_k(\lambda) \Vert z\Vert.$$
\end{itemize}\end{itemize} \end{definition}

Using this notation, the  following theorem is proved in \cite[Theorem 2.7] {todos}:

\medskip

\begin{theorem} \label{castillo} Let $X$ be a Banach space with a boundedly complete Schauder basis and  a near-infinity concentrated (n.i.c.) norm on X. Then $X$ has the FPP. \end{theorem}

\medskip

When  $\mathcal{P}$ is the  projection scheme given in Example \ref{Sch}, one readily verifies that:

\begin{itemize}

\item[(1)] The norm is  sequentially separating if and only if $X$ is of $\mathcal{P}$ asymptotically $\ell_1$-type.

\item[(2)] The norm is premonotone if and only if $\Vert I-P_k\Vert \leq 1$ for every $k\in \mathbb{N}$.

\item[(3)] Condition (3) in Definition \ref{ss} holds if $\Delta(P)<1/(1+2{K}(\mathcal{P}))$.

\item[(4)]  Finally, $X$ is $\mathcal{P}$-AUC, since all projections $P_k$ have finite-dimensional rank.

\end{itemize}

Thus, Theorem \ref{main} recovers and extends the results of \cite{todos} when the projection scheme arises from the canonical projections associated with a Schauder basis.
Nevertheless, as  proved in \cite[Section 7]{alfredo}, a Banach space with a Schauder basis whose norm is sequentially separating has the Schur property and  it is therefore hereditarily $\ell_1$. The former means that every infinite-dimensional subspace contains an isomorphic copy of $\ell_1$. By using general projection schemes, 
we can now assert that fulfilling the Schur property or being hereditarily $\ell_1$ are not necessary conditions for a nonreflexive Banach space to have the FPP. In fact,  Corollary \ref{uciso} shows that every uniformly convex Banach space is isometrically embedded into a nonreflexive Banach space with the FPP. Further examples are provided in next section for the class of variable Lebesgue spaces.

\section{ The FPP for Lebesgue spaces with variable exponent}\label{VLSsection}

In this section, we apply Theorem \ref{main} to variable Lebesgue spaces. Recall that every Banach space containing an isometric copy of $\ell_1$ fails the FPP. 
We will obtain a complete characterization of those variable Lebesgue spaces having the fixed point property: namely, those that precisely do not contain an isometric copy of $\ell_1$. 

\medskip

As explained in Section \ref{pre}, variable Lebesgue spaces constitute a class of Banach function spaces introduced by Nakano \cite{N} and subsequently developed by Orlicz and Musielak \cite{MO}. We refer to Section \ref{pre} for the definition of variable Lebesgue spaces and the notation used throughout this section.
\medskip

We  begin by identifying the situations where the fixed point property necessarily fails.

\begin{lemma} \label{failure} Let $(\Omega,\Sigma,\mu)$ be a $\sigma$-finite  measure space and $p: \Omega\to [1,\infty]$
a
measurable function. Assume that one of the following conditions holds:
\begin{itemize}
\item[(i)]  $p_+(\Omega\setminus \Omega_\infty)=+\infty$;

\item[(ii)]
 The set $\Omega_1$
 either contains a nonatomic set with positive measure or
infinitely many atoms;

\item[(iii)]

 The set $\Omega_\infty$
 either contains a nonatomic set with positive measure
or infinitely many atoms .

\end{itemize}
Then $L^{p(\cdot)}(\Omega)$ fails to have the FPP.\end{lemma}

\begin{proof}
We will check that $\ell_1$  is contained isometrically
in $L^{p(\cdot)}(\Omega)$ whenever one of the above assertions holds. Indeed,
an isometric copy of $\ell_\infty$  can be found in case that (i) holds (see \cite[Theorem 2.3]{DJ}) and $\ell_\infty$  contains an
isometric copy of every separable Banach space. When (ii) holds, an isometric
copy of $L_1([0,1])$ can be given whenever a nonatomic set with positive measure exists
or, simply,  an isometric copy of $\ell_1$ can be formed with infinitely many atoms.
Likewise, we arrive at the same conclusion whenever (iii) holds.\end{proof}

Ruling out  the possible cases where the failure of the FPP is known,  in the rest of this section we are going to
assume  that  $p_+(\Omega\setminus \Omega_\infty) <\infty$ and the set $\Omega_1\cup\Omega_\infty$ only contains finitely many atoms.
Under this setting, we are going to construct a projection scheme for variable Lebesgue spaces that satisfies the requirement of Theorem \ref{main}. In order to do that, we fix a sequence $(\gamma_n)_n$  in $(1,\infty)$ decreasing to 1.

\begin{definition} \label{PSV}  Assume that  $p_+(\Omega\setminus \Omega_\infty) <\infty$ and the set $\Omega_1\cup\Omega_\infty$ only contains finitely many atoms.
Define $P_n(f)=f\cdot 1_{\{\gamma_n\leq p\le \infty\}}$ for all $n\in\mathbb{N}$  and $P_\omega(f)=f\cdot 1_{\{1< p\le \infty\}}$.
\end{definition}

The remaining of this section is devoted to the proof that $\mathcal{P}=\{P_n: 1\le n\le \omega\}$ defines a projection scheme on $L^{p(\cdot)}(\Omega)$ verifying the hypotheses of Theorem \ref{main}.

\begin{lemma} $\mathcal{P}=\{P_n: 1\le n\le \omega\}$ is a  premonotone projection scheme which is boundedly complete.

 \end{lemma}

\begin{proof}

It is clear that $P_n\circ P_m=P_{\min\{n,m\}}$ if $n<m\le \omega$ and that $I-P_\omega$ has rank in a finite-dimensional closed subspace, since $\Omega_1$ is formed  by just finitely many atoms at most.
To prove $\lim_n \V P_n(f)-P_\omega(f)\V=\lim_n \V f\cdot 1_{\{1<p<\gamma_n\}}\V=0$,  observe that $(f\cdot 1_{\{1<p<\gamma_n\}})_n$ is a monotonous sequence converging to zero almost everywhere. Since $p_+(\Omega\setminus \Omega_\infty) <\infty$  the convergence for the modulus implies the convergence for the norm. It is trivial that $\V I-P_n\V\le 1$ for all $n\in\mathbb{N}$.

\medskip

 To prove bounded completeness, consider a bounded sequence  $(f_n)_n$  in  $L^{p(\cdot)}(\Omega)$, say $\Vert f_n\Vert \leq M$ for all $n\in \mathbb{N}$, such that $P_n(f_m)=f_n$ for $m\geq n$.  In particular $P_n(f_n)=f_n$ which implies that $supp(f_n)\subset \{\gamma_n<p\le \infty\}$.
The above implies that, defining $f(t)=f_n(t)$ whether $p(t)\geq \gamma_n$, we obtain a well-defined measurable function on $\{1<p\le +\infty\}$.
Moreover, the sequence $\{P_n(f)\}_n=\{f \cdot 1_{\{\gamma_n<p\le +\infty\}}\}_n$  is nondecreasing and convergent to $f$ almost everywhere. From the Monotone Convergence Theorem and Lemma \ref{properties} we obtain   $$\rho(f/M)=\lim_n \rho(P_n(f)/M)\le \limsup_n\rho(f_n/M) \leq 1$$ and so $f$ belongs to  $L^{p(\cdot)}(\Omega)$.

\end{proof}

\begin{lemma} \label{kUC2}
$L^{p(\cdot)}(\Omega)$ is $\mathcal{P}$-AUC
\end{lemma}
\begin{proof}

Notice that for all $n\in\mathbb{N}$ we can write $P_n=A_n+B_n$ where $A_n(f)=f\cdot 1_{\{\gamma_n\leq p< \infty\}}$ and $B_n(f)=f \cdot 1_{\Omega_\infty}$. It is clear that $B_n$ has finite-rank for every $n\in\mathbb{N}$. Let us prove that $L^{p(\cdot)}(\Omega)$ is $\mathcal{A}$-UC where $\mathcal{A}=\{A_n: n\in\mathbb{N}\}$.

\medskip

 Let  $\rho_n:L^{p(\cdot)}(\Omega)\to [0,\infty]$ be the modular defined by
$$\rho_n(f)=\rho(A_n(f))=\int_{\{\gamma_n\le p <+\infty\}} \vert f(t)\vert^{p(t)}dt.$$ Since $p_+(\Omega\setminus \Omega_\infty)<\infty$,  we can use the uniform convexity of the  modular $\rho_n$ given by Lemma \ref{ucc2}. Hence, for a given  $\delta>0$, there exists some $\alpha_n>0$ (depending on $\gamma_n, \delta $ and $p_+(\Omega\setminus \Omega_\infty))$ such that
$$
\rho_n\left(\frac{f+g}{ 2}\right)\le \frac{\rho_n(f)+\rho_n(g)}{ 2}-\alpha_n
$$
whenever $\Vert A_n(f-g)\Vert>\delta$ and $\Vert A_n(f)\Vert\le \Vert f\Vert \le 1$ and $\Vert A_n(g)\Vert\le \Vert g\Vert\le 1$.

\vskip 0.5cm
Using the additivity of the modular   and Proposition \ref{properties} we deduce
$$
\rho\left(\frac{f+g}{ 2}\right)\le \frac{\rho(f)+\rho(g)}{2}-\alpha_n\le 1-\alpha_n.
$$
and
$$\left\V \frac{f+g}{ 2}\right\V\le (1-\alpha_n)^{\frac{1}{ p_+(\Omega\setminus \Omega_\infty)}}=1-(1-\alpha_n)^{\frac{1}{ p_+(\Omega\setminus \Omega_\infty)}},
$$
and the proof is complete.

\end{proof}

Our next goal is to show that $L^{p(\cdot)}(\Omega)$ is of $\mathcal{P}$-asymptotically $\ell_1$-type and that $\Delta(\mathcal{P})<{1\over 3}$ since ${K}(\mathcal{P})=1$.
To this end, we first establish a simple lemma that will be used in the proof of Lemma \ref{SS2}:

\begin{lemma}\label{gammaa} Assume $1 < \gamma<\infty$ and let $f,g\in L^{p(\cdot)}(\Omega)$  be normalized disjointedly supported functions with support in $\{1\le p\leq \gamma\}$. Then:
$$
\V f+g\V\ge 2^{\frac{1}{\gamma}}
$$
\end{lemma}
\begin{proof}
Write $r=2^{1/\gamma}$. Since $r>1$ we have
$$ \rho\left(\frac{g}{r}\right)=\int_\Omega \left(\frac{1}{r}\right)^{p(t)} \vert g(t)\vert^{p(t)} d\mu \geq  \left(\frac{1}{r}\right)^\gamma\int_\Omega \vert g(t)\vert^{p(t)} d\mu
=\left(\frac{1}{r}\right)^\gamma.$$
Likewise $\rho(f)\geq  \left(\frac{1}{r}\right)^\gamma$ and
since $f$ and $g$ are disjointly supported we have
$$
\rho\left(\frac{f+g}{ r}\right)=\rho\left(\frac{f}{ r}\right) +\rho\left(\frac{g}{ r}\right)\geq
\left(1 \over r\right)^{\gamma} +\left(1\over r\right)^{\gamma}=1
$$
and, using Lemma \ref{properties},  the proof is complete.
\end{proof}

\begin{lemma}\label{SS2}  $L^{p(\cdot)}(\Omega)$  is  of $\mathcal{P}$-asymptotically $\ell_1$-type.
\end{lemma}

\begin{proof} Fix  $k\in \mathbb{N}$. We will prove that $S(P_k)\geq 2^{\frac{1}{\gamma_k}}$ and
 passing to the limit as  $k\to \infty$,
we obtain that $S(\mathcal{P})=2$. Consider any $\{(f_n),f\}\in  \mathcal{S}_1(X) \times S_X$ with $X=L^{p(\cdot)}(\Omega)$ and $P_k(f) = 0$. That means, in particular, that supp$\,(f)\subset \{1<t<\gamma_k\}\cup \Omega_1$ and $\lim_n\V f_n\cdot 1_{\Omega_1}\V=0$ (see Definition \ref{Pconvergence}).

\medskip

Note that  the following equality holds:
$$
\limsup_n\left\V f_n-\frac{(I-P_k)(f_n)}{\Vert (I-P_k)(f_n)\Vert}\right\V=0,
$$
and
$$\limsup_n\Vert  f_n+f\Vert = \limsup_n\left \Vert \frac{(I-P_k)(f_n)}{\Vert (I-P_k)(f_n)\Vert} +f\right\Vert\ge 2^{1\over \gamma_k}
$$
 following the inequality in Lemma \ref{gammaa}.
\end{proof}

\begin{lemma}\label{lambda2}
Given $\gamma>1$, the following holds:  for all normalized $f,g\in L^{p(\cdot)}(\Omega)$  disjointly supported functions with $supp(g)\subset\{\gamma\le p\le \infty \}$ we have
$$
\V f+\lambda g\V\le 1+\lambda^\gamma
$$
for $\lambda\in (0,1)$.

\end{lemma}

\begin{proof}
Since $p_+(\Omega\setminus \Omega_\infty)<\infty$,   we have $\rho(f)=1=\rho(g)$ according to Lemma \ref{properties}. Set $r:=1+\lambda^\gamma>1$. Then, we have
$$
\rho\left(\frac{f+\lambda g}{r}\right)=\rho\left(\frac{f}{ r} \right)+\rho\left(\frac{\lambda g}{r}\right)\le
 \frac{1}{ r}\rho(f)+ \frac{1}{ r}\rho(\lambda g)\le  \frac{1}{ r}(1+\lambda^\gamma)   =1.
$$

\end{proof}

\vskip 0.5cm


\begin{lemma} For all $k\in \mathbb{N}$, we have $\Delta_{P_k}(\lambda)\leq \lambda^{\gamma_k}$. In particular  $\Delta(\mathcal{P})=0$.
\end{lemma}

\begin{proof}

Assume that $\{(f_n),f\}\in \mathcal{S}_1(X)\times S_X$ and $P_k(f)=f$. We have that supp$\, (f)\subset \{\gamma_k\le p\le \infty\}$ and supp$\, (I-P_k)(f_n) \subset \{1<p<\gamma_k\}$. By Lemma  \ref{lambda2} and the $\mathcal{P}$-convergence we have:
$$\limsup_n \Vert f_n+\lambda f\Vert = \limsup_n \Vert (I-P_k)(f_n)+\lambda f\Vert  \leq 1+\lambda^{\gamma_k}.
$$
The above implies that $\lim_{\lambda\to 0^+}{\Delta_{P_k}(\lambda)\over \lambda}=0$ and this holds for every $k\in\mathbb{N}$ so $\Delta(\mathcal{P})=0$.
\end{proof}

As a conclusion of Theorem \ref{main} and the properties of the projection scheme defined on variable Lebesgue space we finally deduce:
\begin{theorem}\label{main2}
Assume that $(\Omega,\Sigma, \mu)$ is a $\sigma$-finite measure space $p_+(\Omega\setminus\Omega_\infty) <\infty$ and the set $\Omega_1\cup\Omega_\infty$ only contains finitely many atoms. Then $L^{p(\cdot)}(\Omega)$ has the FPP.

\end{theorem}

Combining Theorems \ref{main2} and \ref{failure},  we obtain a complete characterization of the FPP in Lebesgue spaces with variable exponents,  extending the results obtained in \cite{DJ}, where it was proved: A variable Lebesgue space $L^{p(\cdot)}(\Omega)$ fulfills the  weak fixed point property if and only if  it does not contain an  isometric copy of $L_1[0,1]$.

Additionally, a Banach space containing an asymptotically isometric copy of $\ell_1$ fails the FPP \cite{DL} and it was proved in \cite[Section 5]{DJ} that variable Lebesgue spaces containing an asymptotically isometric copy of $\ell_1$ must contain an isometric copy of $\ell_1$. This fact  identifies  a difference with the results obtained in \cite{DL} for nonreflexive closed subspaces of $L_1[0,1]$, which  fail the FPP because of  all of them contain asymptotically isometric copies of $\ell_1$, but not isometric copies of $\ell_1$ necessarily.  Now the announced characterization is the following:

\begin{corollary}\label{co} Assume that $(\Omega,\Sigma, \mu)$ is a $\sigma$-finite measure space. Let $p:\Omega\to [1,\infty]$ be a measurable function. The following conditions are equivalent:

\begin{itemize}

\item[1)] $L^{p(\cdot)}(\Omega)$ has the FPP.

\item[2)] $L^{p(\cdot)}(\Omega)$ does not contain isometrically $\ell_1$.

\item[3)] $p_+(\Omega\setminus \Omega_\infty) <\infty$ and the set $\Omega_1\cup\Omega_\infty$ only contains finitely many atoms.

\end{itemize}
\end{corollary}

\section{Consequences and Further Applications}

\medskip

\textbf{1. Application to Nakano sequence spaces.}

\medskip

Following the techniques introduced in \cite{todos}, it was proved in \cite[Corollary 4.7]{DJ} that the Nakano space $\ell^{p_n}$  has the FPP whenever $\lim_np_n=1$. The authors raised in \cite{DJ} the following question: Consider  a bounded sequence $(p_n)\subset (1,+\infty)$ with $\liminf_n p_n=1$ (which implies that the Nakano space $\ell^{p_n}$ is nonreflexive).  Does the space $\ell^{p_n}$  have the FPP?

  \medskip

 Corollary \ref{co} gives an affirmative answer to this question. Indeed, every bounded sequence  $(p_n)\subset (1,+\infty)$ with $\liminf_n p_n=1$  satisfies condition (3) in  Corollary \ref{co}. Notice that, in constrast with the stardard Schauder projections used in \cite[Corollary 4.7]{DJ}, we use in Corollary \ref{co} the projections $P_n(x)=\sum_{n\in A_n} x(n)e_n$ where $A_n=\{n\in\mathbb{N}: 1<p_n<\gamma_n\}$ where $(\gamma_n)$ is a nondecresing sequence in $(1,+\infty)$ with $\lim_n \gamma_n=1$ chosen beforehand.

\medskip

  \textbf{2. Examples of nonreflexive Banach function spaces with the FPP.}

     \medskip

Corollary \ref{co} also provides a large family of nonreflexive Banach spaces with the FPP that contain infinite-dimensional reflexive subspaces, none of which  is isomorphic to $\ell_1$. In fact,
 there exists a nonreflexive variable Lebesgue space $L^{p(\cdot)}(\Omega)$ which satisfies the FPP and
contains isometric copies of infinitely many spaces  $L^{p_n}([0,1])$, $1<p_n<\infty$. As an illustration, consider the following example:

\medskip

Assume that $\{p_n\}$ is a  bounded sequence in $(1,\infty)$ such that $\liminf_n p_n=1$.  Assume $\Omega=(0,1)$ and define $p:\Omega\to (1,\infty)$ as the step function $p(t)=p_n$ if $t\in (1/2^{n},1/2^{n-1})$. Then,  $L^{p(\cdot)}(\Omega)$ contains, as isometric subspaces, all the spaces  $L^{p_n}((1/2^{n},1/2^{n-1}))$, $n=1,2,..$ that are isometric to  $L^{p_n}([0,1])$.  Because $p_- =1$ the space $L^{p(\cdot)}(\Omega)$  is nonreflexive \cite[Corollary 2.8]{cruz} and it satisfies the FPP because it satisfies condition 2) or 3) in Corollary \ref{co}. Thus, we have constructed  a  nonreflexive space with the FPP which contains isometric copies of infinitely many reflexive $L^p([0,1])$ spaces.


\medskip

\textbf{3. Extension to mappings of type (L).}

\medskip

For any bounded  sequence $(x_n)$ in a Banach space $X$ and any $\rho>0$, the level set $A(x_n,\rho)$ is defined by
$$ A(x_n,\rho)=\{x\in C: \limsup_n\Vert x_n-x\Vert\leq \rho\}.$$
It is clear that $A(x_n,\rho)$ is a (possibly empty)  closed, bounded and convex subset of $C$.  The proof of Theorem \ref{main} shows that the nonexpansiveness assumption is only used through the following two properties.

(1) There exists an a.f.p.s. for $T$ in every convex, closed, bounded, $T$-invariant subset of the domain  (used in Lemma \ref{afps}).

(2) $A(x_n,\rho)$ is $T$-invariant whenever $(x_n)$ is an a.f.p.s. 

\medskip

Consequently, Theorem \ref{main} and Theorem \ref{main2} together  with Corollary \ref{co} remain valid for every continuous mapping satisfying (1) and (2). This includes, for instance,  the class of mappings of type (C) defined by Suzuki \cite{Suz}, and more generally  mappings of $(L)$-type, introduced in \cite{LLMo}. In particular, all results in \cite{MoLL} about existence  of fixed point for mappings of $(L)$-type in nonreflexive spaces follow from Theorem \ref{main}.

\end{document}